\documentclass[12pt]{amsart}

\usepackage[margin=2.4cm]{geometry}
\usepackage{a4wide}

\usepackage{amsfonts,amsmath,amssymb,amsthm}
\usepackage{enumerate}

\newtheorem{theorem}{Theorem}[section]

\newtheorem{lemma}[theorem]{Lemma}

\makeatletter
\renewcommand{\@makefntext}[1]{#1}
\makeatother

\begin{document}

\author{Rafa{\l} Kapica, Janusz Morawiec}
\email[R. Kapica]{rkapica@math.us.edu.pl}
\title[Integrable solutions of inhomogeneous refinement equations]{Integrable solutions of inhomogeneous refinement type equations on intervals}
\email[J. Morawiec]{morawiec@math.us.edu.pl}
\date{}
\address{Institute of Mathematics, University of Silesia, Bankowa 14, 40-007 Katowice, Poland}


\begin{abstract}
Given a probability measure $P$ on a $\sigma$-algebra of subsets of a set $\Omega$, an interval $I\subset\mathbb R$, $g\in L^1(I)$, and a function $\varphi\colon I\times\Omega\to I$ fulfilling some conditions we obtain results on the existence of solutions $f\in L^1(I)$ of the inhomogeneous refinement type equation 
$$
f(x)=\int_{\Omega}\big|\varphi'_x(x,\omega)\big|f(\varphi(x,\omega))dP(\omega)+g(x).
$$
\end{abstract}

\maketitle


\section{Introduction}

A discrete inhomogeneous refinement equation 
$$
f(x)=\sum_{n\in\mathbb Z}c_nf(kx-n)+g(x)
$$ 
as well as its continuous counterpart 
$$
f(x)=\int_{\mathbb R}f(kx-y)d\mu(y)+g(x)
$$ 
were studied in many papers. For example, the discrete inhomogeneous refinement equation has been used in \cite{GHM1994} for a construction of  multiwavelets from a "fractal equation", in \cite{SN1996} for a construction of boundary wavelets, in \cite{SZ1998} for the study of convergence of cascade algorithms and in \cite{JL2012} for the study of quad/triangular subdivision. A complete uniform characterization of the existence of distributional solutions of both the above classical inhomogeneous refinement equations  was obtained in \cite{JJS2000} to cover all cases of interest. 

A common extension of both the above inhomogeneous refinement equations is the following  inhomogeneous refinement type equation
\begin{equation}\label{e}
f(x)=\int_\Omega|\varphi'_x(x,\omega)|f(\varphi(x,\omega))P(d\omega)+g(x).
\end{equation}
Integrable solutions of the poly-scale version of equation (\ref{e}) have been recently investigated by the use of the Banach fixed point theorem  in \cite{KM}. In the present paper we are interested in integrable solutions $f\colon I\to\mathbb R$ of equation (\ref{e}) assuming that $(\Omega,\mathcal A,P)$ is a  complete probability space, $I\subset\mathbb R$ is an open interval (finite or infinite; possible equals to the whole real line), $g\in L^1(I)$ and  $\varphi\colon I\times\Omega\to I$ is a function satisfying the following conditions:
\begin{itemize}
\item[(a)] $\varphi(\cdot,\omega)$ is a diffeomorphism from $I$ onto $I$ for every $\omega\in\Omega$; 
\item[(b)] $\varphi(x,\cdot)$ is $\mathcal A$-measurable for every $x\in I$;
\item[(c)] $(\lambda\otimes  P)(\varphi^{-1}(B))=0$ for every Borel set $B\in{\mathcal B}(I)$ of Lebesgue measure $\lambda$ equals zero.
\end{itemize}

The reader may ask why we choose the interval $I$ to be open. No special reason. In fact all the presented below results can be reformulated if we start with $I$ to be closed or closed on one side.


\section{Preliminaries} 

We begin by observing that if $g\colon I\to\mathbb R$ is a representative of $g\in L^1(I)$ and $f,\widetilde{f}\colon I\to\mathbb R$ are two representatives of a function from $L^1(I)$ such that $(\ref{e})$ holds for almost all $x\in I$, then $\widetilde{f}$ also satisfies (\ref{e}) for almost all $x\in I$; see \cite{KM2008} for details in the homogeneous case of equation (\ref{e}) with $I=\mathbb R$. This observation allows us to accept the following definition: A function $f\in L^1(I)$ is said to be an $L^1(I)$-solution of equation (\ref{e}), if every representative of $f$ satisfies (\ref{e}) for almost all $x\in I$ with respect to the Lebesgue measure.

Let $X$ be a separable metric space. We say that $\psi\colon X\times\Omega\to X$ is a random-valued function, if it is measurable with respect to the product $\sigma$-algebra $\mathcal B(X)\otimes\mathcal  A$. Following \cite{BK1977}, we define the sequence $(\psi^n)_{n\in\mathbb N}$ of iterates of a random-valued function $\psi\colon X\times\Omega\to X$ as follows:
$$
\psi^1(x,\varpi)=\psi(x,\omega_1)\hspace{5ex}\hbox{and}\hspace{5ex}
\psi^{n+1}(x,\varpi)=\psi(\psi^n(x,\varpi),\omega_{n+1})
$$
for all $x\in X$ and $\varpi=(\omega_1,\omega_2,\dots)\in\Omega^{\infty}$.  Since $\psi^n(\cdot,\varpi)$ depends only on the first $n$ coordinates of $\varpi\in\Omega^\infty$, we may consider the iterate $\psi^n$ as a function defined on $X\times\Omega^{\infty}$ or, alternatively, on $X\times\Omega^n$. 
The basic property of iterates of rv-functions says that they are random-valued functions on the product probability space $(\Omega^\infty,{\mathcal A}^\infty, P^\infty)$.

It turns out that the above defined iterates form a random dynamical system (see \cite{A1995}) and a Markov chain (see \cite{MT1993}). Moreover, iteration is the fundamental technique for solving functional equations, and iterates usually appear in the formulas for solutions (see \cite{KCG1990}). In this paper we will apply a result on the convergence in law of the sequence of iterates of a random-valued function  obtained in \cite{B2009}.


\section{Compactly supported $L^1(\mathbb R)$-solutions} 

In many applications compactly supported solutions of inhomogeneous, as well as homogeneous, refinement equations play an important role. Compactly supported distributional solutions of inhomogeneous refinement equations were considered, among others,  in \cite{DH1998, DH1999, S2001}. In this section we are interested in compactly supported $L^1(\mathbb R)$-solutions of equation (\ref{e}). We begin with the following useful lemma.

\begin{lemma}\label{lem2}
Let $J\subset\mathbb R$ be an open interval and let $\alpha\colon J\to I$ be a diffeomorphism  $($onto $I$$)$. Then:
\begin{itemize}
\item[\rm (i)] The function $\phi\colon J\times\Omega\to J$ given  by 
\begin{equation}\label{phi}
\phi(x,\omega)=\alpha^{-1}(\varphi(\alpha(x),\omega))
\end{equation} 
satisfies conditions {\rm (a)--(c)} with $I$ replaced by $J$;
\item[\rm (ii)] If $f$ is an  $L^1(I)$-solution of equation $(\ref{e})$, then $\widetilde{f}=|\alpha'|\!\cdot\! f\circ\alpha$ is an $L^1(J)$-solution of the equation
\begin{equation}\label{e0}
\widetilde{f}(x)=\int_\Omega|\phi'_x(x,\omega)|\widetilde{f}(\phi(x,\omega))P(d\omega)+\widetilde{g}(x),
\end{equation}
where $\widetilde{g}=|\alpha'|\!\cdot\! g\circ\alpha$.
\end{itemize}
\end{lemma}

\begin{proof}
Assertion (i) is clear.  To prove assertion (ii) assume that $f$ satisfies (\ref{e}). Clearly, $\widetilde{f},\widetilde{g}\in L^1(J)$. Moreover,
\begin{eqnarray*}
\widetilde{f}(x)&=&|\alpha'(x)|f(\alpha(x))=\int_\Omega|\varphi'_x(\alpha(x),\omega)||\alpha'(x)|f(\varphi(\alpha(x),\omega))P(d\omega)+|\alpha'(x)|g(\alpha(x))\\
&=&\int_\Omega|\phi'_x(x,\omega)| |\alpha'(\phi(x,\omega))|f(\alpha(\phi(x,\omega)))P(d\omega) +\widetilde{g}(x)\\
&=&\int_\Omega|\phi'_x(x,\omega)|\widetilde{f}(\phi(x,\omega))P(d\omega)+\widetilde{g}(x),
\end{eqnarray*}
and the proof is complete.
\end{proof}

Now we are in a position to prove the following result.

\begin{theorem}
Assume $I=\mathbb R$. Let $F\subset\mathbb R$ be a closed and non-degenerated interval such that 
\begin{equation}\label{supp}
{\rm supp}\,g\subset F
\end{equation}
and 
\begin{equation}\label{varphi}
\varphi({\rm int}F\times\{\omega\})={\rm int}F\hspace{5ex}\hbox{for almost all }\omega\in\Omega.
\end{equation} 
Then equation $(\ref{e})$ has a compactly supported $L^1(\mathbb R)$-solution with ${\rm supp}f\subset F$ if and only if there exists a diffeomorphism $\alpha\colon {\rm int}F\to \mathbb R$ such that equation $(\ref{e0})$ with $\phi\colon \mathbb R\times\Omega\to \mathbb R$ given by $(\ref{phi})$ has an $L^1(\mathbb R)$-solution.
\end{theorem}

\begin{proof}
Assume that $f$ is an $L^1(\mathbb R)$-solution of equation $(\ref{e})$ with ${\rm supp}f\subset F$.
 
Clearly, $f|_{{\rm int}F}\in L^1({\rm int}F)$. Moreover, for almost all $x\in {\rm int}F$ we have
\begin{eqnarray*}
f|_{{\rm int}F}(x)&=&f(x)=\int_\Omega|\varphi'_x(x,\omega)|f(\varphi(x,\omega))P(d\omega)+g(x)\\
&=&\int_\Omega|\varphi'_x(x,\omega)|f|_{{\rm int}F}(\varphi(x,\omega))P(d\omega)+g|_{{\rm int}F}(x),
\end{eqnarray*}
which shows that $f|_{{\rm int}F}$ is an $L^1({\rm int}F)$-solution of equation $(\ref{e})$. Finally, applying Lemma \ref{lem2} with $I={\rm int}F$ and $J=\mathbb R$ we infer that equation $(\ref{e0})$, with $\phi\colon \mathbb R\times\Omega\to \mathbb R$ given by $(\ref{phi})$, has an $L^1(\mathbb R)$-solution.

To prove the converse implication assume that  there is a diffeomorphism $\alpha\colon{\rm int}F\to \mathbb R$ such that equation $(\ref{e0})$, with $\phi\colon \mathbb R\times\Omega\to \mathbb R$ given by $(\ref{phi})$,  has an $L^1(\mathbb R)$-solution. 

Using Lemma \ref{lem2} with $I={\rm int}F$ and $J=\mathbb R$ once again we conclude that equation $(\ref{e})$ has an $L^1({\rm int}F)$-solution $\widetilde{f}$. 

Define a function $f\colon\mathbb R\to\mathbb R$ by putting 
$$
f(x)=\left\{\begin{array}{ccc}
\widetilde{f}(x)&\hbox{for}&x\in {\rm int} F,\\
0&\hbox{for}&x\in\mathbb R\setminus{\rm int} F.
\end{array}
\right.
$$
Clearly, $f\in L^1(\mathbb R)$, and for almost all $x\in{\rm int}F$ we have 
\begin{eqnarray*}
f(x)&=&\widetilde{f}(x)=\int_\Omega|\varphi'_x(x,\omega)|\widetilde{f}(\varphi(x,\omega))P(d\omega)+g(x)\\
&=&\int_\Omega|\varphi'_x(x,\omega)|f(\varphi(x,\omega))P(d\omega)+g(x).
\end{eqnarray*}
If $x\in\mathbb R\setminus{\rm int}F$, then (a) and (\ref{varphi}) imply  $\varphi(x,\omega)\in\mathbb R\setminus{\rm int}F$ for almost all $\omega\in\Omega$. Hence $f(x)=0$ and $f(\varphi(x,\omega))=0$  for all $x\in\mathbb R\setminus{\rm int}F$ and almost all $\omega\in\Omega$. In consequence, (\ref{e}) holds for almost all $x\in\mathbb R\setminus{\rm int}F$, by (\ref{supp}).

The proof is complete.
\end{proof}


\section{$L^1(I)$-solutions as Radon-Nikodym derivatives} 

It is known that integrable solutions of the homogeneous refinement type equation are determined, up to a  multiplicative constant, by the Radon-Nikodym derivative of their integrals over $(-\infty,x]$, where $x\in\mathbb R$, with respect to the one dimensional Lebesgue measure (see \cite[Section 3.4]{KM2013}). We will show that also $L^1(I)$-solutions  of equation (\ref{e}) have such a property. For this purpose fix $x_0\in I$ and put
$$
\Omega_+=\left\{\omega\in\Omega:\varphi'_x(x_0,\omega)>0\right\}
\hspace{5ex}\hbox{and}\hspace{5ex}
\Omega_-=\left\{\omega\in\Omega:\varphi'_x(x_0,\omega)<0\right\}.
$$
From assumption (a) we see that the sets $\Omega_+$ and $\Omega_-$ do not depend on the choice of $x_0$ and $\Omega=\Omega_+\cup\Omega_-$. Next define a function $G\colon I\to\mathbb R$ putting
\begin{equation}\label{G}
G(x)=\int_{\inf I}^x g(t)dt.
\end{equation}

The following extension of Proposition 3.1 from \cite{KM2008} is a useful tool for studying the existence of $L^1(I)$-solutions of equation (\ref{e}).

\begin{lemma}\label{lem1}
Assume $f\in L^1(I)$ and put $\alpha=\int_If(x)dx$.
\begin{itemize}
\item[\rm (i)]  If $F\colon I\to\mathbb R$ given by
\begin{equation}\label{2}
F(x)=\int_{\inf I}^xf(t)dt
\end{equation}
satisfies
\begin{equation}\label{e1}
F(x)=\int_{\Omega_+}F(\varphi(x,\omega)) P(d\omega)+\int_{\Omega_-}[\alpha-F(\varphi(x,\omega))] P(d\omega)+G(x)
\end{equation} 
for every  $x\in I$, then $f$ is an $L^1(I)$-solution of equation $(\ref{e})$.
\item[\rm (ii)] If $f$ is an $L^1(I)$-solution of equation $(\ref{e})$, then the function $F\colon I\to\mathbb R$ given by $(\ref{2})$ satisfies $(\ref{e1})$ for every  $x\in I$.
\end{itemize}
\end{lemma}

\begin{proof}
(i) Assume that $F\colon I\to\mathbb R$ given by (\ref{2}) satisfies (\ref{e1}) for every $x\in I$. 
Then 
\begin{eqnarray*}
\int_{\inf I}^xf(t)dt&=&F(x)=\int_{\Omega_+}\int_{\inf I}^{\varphi(x,\omega)}f(t)dtP(d\omega)\\
&&+\int_{\Omega_-}\left[\int_{\inf I}^{\sup I} f(t)dt- \int_{\inf I}^{\varphi(x,\omega)}f(t)dt \right]P(d\omega)+G(x)\\
&=&\int_{\Omega_+}\int_{\inf I}^{\varphi(x,\omega)}f(t)dt P(d\omega)+
\int_{\Omega_-}\int_{\varphi(x,\omega)}^{\sup I} f(t)dt P(d\omega)+G(x)\\
&=&\int_{\Omega_+}\int_{\inf I}^x\varphi'_x(t,\omega)f(\varphi(t,\omega))dtP(d\omega)\\
&&-\int_{\Omega_-}\int_{\inf I}^x\varphi'_x(t,\omega)f(\varphi(t,\omega))dtP(d\omega)+G(x)\\
&=&\int_{\Omega}\int_{\inf I}^x |\varphi'_x(t,\omega)|f(\varphi(t,\omega))dtP(d\omega)+\int_{\inf I}^x g(t)dt\\
&=&\int_{\inf I}^x\left[\int_{\Omega}|\varphi'_x(t,\omega)|f(\varphi(t,\omega))P(d\omega)+g(t)\right]dt
\end{eqnarray*}
for every $x\in I$. Hence $f$ coincides, in $L^1$ sense, with $\int_{\Omega}|\varphi'_x(\cdot,\omega)|f(\varphi(\cdot,\omega))P(d\omega)+g$. This means that $f$ is an $L^1(I)$-solution of equation (\ref{e}).

(ii) If $f$ is an $L^1(I)$-solution of equation (\ref{e}), then arguing as above we can show that (\ref{e1}) holds with $F\colon I\to\mathbb R$ defined by (\ref{2}).
\end{proof}


\section{$L^1(I)$-solutions in the case where $P(\Omega_+)\in\{0,1\}$} 

In this section we will use a result on the convergence in law of a sequence of iterates of a random-valued function obtained in \cite{B2009}. Thus we assume that there exists $l\in (0,1)$ such that
\begin{equation}\label{a}
\int_\Omega |\varphi(x,\omega)-\varphi(y,\omega)|P(d\omega)\leq l|x-y|\hspace{3ex}
\hbox{ for all }x, y\in I
\end{equation}
and
\begin{equation}\label{b}
\int_\Omega |\varphi(x,\omega)-x|P(d\omega)<+\infty \hspace{3ex}\hbox{ for every }x\in I.
\end{equation}

Note that (\ref{a}) and (a) imply unboundedness of the interval $I$.

Assume that $\inf I\in\mathbb R$ and $\sup I=+\infty$. Then, by (a), we have $\lim_{x\to \inf I}\varphi(x,\omega)\in\{\inf I,+\infty\}$ for every $\omega\in\Omega$. Put
$$
A=\left\{\omega\in\Omega:\lim_{x\to \inf I}\varphi(x,\omega)=+\infty\right\}.
$$
Fixing $y\in I$ and making use of the Lebesgue monotone convergence theorem and (\ref{a}) we obtain
\begin{eqnarray*}
\int_A\lim_{x\to\inf I} |\varphi(x,\omega)-\varphi(y,\omega)|P(d\omega)&=&
\lim_{x\to\inf I}\int_A |\varphi(x,\omega)-\varphi(y,\omega)|P(d\omega)\\
&\leq&\lim_{x\to\inf I}\int_\Omega |\varphi(x,\omega)-\varphi(y,\omega)|P(d\omega)\\
&\leq& l|\inf I-y|<+\infty,
\end{eqnarray*}
and hence $P(A)=0$. In a similar way we can prove that $P(\{\omega\in\Omega:\lim_{x\to \sup I}\varphi(x,\omega)=-\infty\})=0$ in the case where $\inf I=-\infty$ and $\sup I\in\mathbb R$.

The above calculation shows that there is no loss of generality in assuming that $\lim_{x\to \inf I}\varphi(x,\omega)\in\mathbb R$ for every $\omega\in\Omega$ in the case where $\inf I\in\mathbb R$ and $\lim_{x\to \sup I}\varphi(x,\omega)\in\mathbb R$ for every $\omega\in\Omega$ in the case where $\sup I\in\mathbb R$.

Define a function $\varphi_0\colon{\rm cl }I\times\Omega\to{\rm cl}I$ by putting $\varphi_0=\varphi$ in the case where $I=\mathbb R$, or
$$
\varphi_0(x,\omega)=\left\{
\begin{array}{ccc}
\varphi(x,\omega)&\hbox{for}&(x,\omega)\in I\times\Omega,\\
\lim_{y\to x}\varphi(y,\omega)&\hbox{for}&(x,\omega)\in({\rm cl} I\setminus I)\times\Omega,
\end{array}
\right.
$$
in the case where $I\neq\mathbb R$. Note that both $\varphi$ and $\varphi_0$ are random-valued functions, and $\varphi_0(x,\omega)\in{\rm cl}I\setminus I$ for all $x\in{\rm cl}I\setminus I$ and $\omega\in\Omega$. Moreover, (\ref{a}) and the Fatou lemma imply
\begin{equation}\label{aa}
\int_\Omega |\varphi_0(x,\omega)-\varphi_0(y,\omega)|P(d\omega)\leq l|x-y|\hspace{3ex}
\hbox{ for all }x, y\in {\rm cl}I,
\end{equation}
which jointly with (\ref{b}) gives
\begin{equation}\label{bb}
\int_\Omega |\varphi_0(x,\omega)-x|P(d\omega)<+\infty \hspace{3ex}\hbox{ for every }x\in {\rm cl}I.
\end{equation}

Denote by $\pi_n(x,\cdot)$ the distribution of $\varphi_0^n(x,\cdot)$, i.e., 
$$
\pi_n(x,B)=P^\infty(\{\varpi\in\Omega^\infty:\varphi_0^n(x,\varpi)\in B\})
$$ 
for all  $n\in\mathbb N$, $x\in \mathbb R$ and  $B\in{\mathcal B}({\rm cl}I)$.

The following result on the convergence in law of the sequence $(\varphi_0^n(x,\cdot))_{n\in\mathbb N}$ will be our main tool.

\begin{theorem}{\rm (see \cite[Theorem 3.1]{B2009})}\label{thm1}
Assume $(\ref{bb})$ and let $(\ref{aa})$ hold with some $l\in(0,1)$. Then there exists a distribution $\pi$ on ${\rm cl}I$ such that for every $x\in{\rm cl}I$ the sequence $(\pi_n(x,\cdot))_{n\in\mathbb N}$ converges weakly to $\pi$.
\end{theorem}

To formulate our first result of this section let us denote by $D$ the probability distribution function of the distribution $\pi$ obtained in Theorem \ref{thm1}, i.e., 
$$
D(t)=\pi([\inf I,t]\cap\mathbb R)
$$
 for every $t\in{\rm cl}I$.
 
\begin{theorem}\label{thm2}
Assume $P(\Omega_+)=1$. Let $(\ref{b})$ hold and let $(\ref{a})$ be satisfied  with some $l\in (0,1)$. If 
\begin{equation}\label{D}
\int_{I}g(t)(1-D(t))dt=0
\end{equation}
and if there exists $L\in(0,+\infty)$ such that 
\begin{equation}\label{gg}
\left|\int_x^y g(t)dt\right|\leq L|x-y|\hspace{3ex}\hbox{ for all }x,y\in I, 
\end{equation}
then equation $(\ref{e1})$ has an absolutely continuous solution $F\colon I\to\mathbb R$.

Moreover, if there exists $f\in L^1(I)$ such that $(\ref{2})$ holds for every $x\in I$, then $f$ is an $L^1(I)$-solution of equation $(\ref{e})$.
\end{theorem}

\begin{proof}
We first observe that it is enough to show that equation
\begin{equation}\label{e2}
F(x)=\int_{\Omega}F(\varphi_0(x,\omega)) P(d\omega)+G(x),
\end{equation}
where $G\colon{\rm cl}I\to\mathbb R$ is defined by (\ref{G}), has an absolutely continuous solution $F\colon{\rm cl}I\to\mathbb R$; indeed, if $F$ is such a solution, then $F|_I$ is an absolutely continuous solution of equation (\ref{e1}).

Fix $x_0\in {\rm cl}I$. We will show that
\begin{equation}\label{9}
\lim_{n\to\infty}\int_{\Omega^\infty}\int_{\inf I}^{\varphi_0^n(x_0,\varpi)}g(t)dtP^\infty(d\varpi)=0.
\end{equation}

For every $n\in\mathbb N$ denote by $D_n$ the probability distribution function of the distribution $\pi_n(x_0,\cdot)$, i.e., 
$$
D_n(t)=P^\infty(\{\varpi\in\Omega^\infty:\varphi_0^n(x_0,\varpi)\leq t\})
$$
 for every $t\in{\rm cl}I$. 
 
 Applying the Fubini theorem we obtain
\begin{eqnarray*}
\int_{\Omega^\infty}\int_{\inf I}^{\varphi_0^n(x_0,\varpi)}\!\!\!g(t)dtP^\infty(d\varpi)
&=&\int_{\{(t,\varpi)\in{\rm cl}I\times\Omega^\infty: t< \varphi_0^n(x_0,\varpi)\}}g(t)dt P^\infty(d\varpi)\\
&=&\int_I\int_{\{\varpi\in\Omega^\infty:t< \varphi_0^n(x_0,\varpi)\}}g(t)P^\infty(d\varpi)dt\\
&=&\int_Ig(t)\big[1-P^\infty(\{\varpi\in\Omega^\infty: \varphi_0^n(x_0,\varpi)\leq t\})\big]dt\\
&=&\int_Ig(t)dt -\int_Ig(t)D_n(t)dt
\end{eqnarray*}
for every $n\in\mathbb N$. From Theorem \ref{thm1} we conclude that for every $t\in{\rm cl}I$, being a point of continuity of $D$, we have $\lim_{n\to\infty}D_n(t)=D(t)$. Hence $\lim_{n\to\infty}g(t)D_n(t)=g(t)D(t)$ for almost all $t\in{\rm cl}I$, and by the Lebesgue dominated convergence theorem and (\ref{D}) we obtain
$$
\lim_{n\to\infty}\int_I g(t)D_n(t)dt=\int_I g(t)D(t)dt=\int_I g(t)dt.
$$
In consequence, (\ref{9}) holds. 

Finally, according to \cite[Corollary 4.1 (iii)]{B2009} equation (\ref{e2}) has  an absolutely continuous solution $F\colon{\rm cl}I\to\mathbb R$. 

The moreover statement follows from Lemma \ref{lem1}.
\end{proof}

According to \cite[Theorem 3.1 (ii)]{BKM} we have the following counterpart of Theorem \ref{thm2}.

\begin{theorem}\label{thm3}
Assume $P(\Omega_+)=0$. Let $(\ref{b})$ hold and let $(\ref{a})$ be satisfied  with some $l\in (0,1)$. If  there exists $L\in(0,+\infty)$ such that $(\ref{gg})$ holds, then equation $(\ref{e1})$ has an absolutely continuous solution $F\colon I\to\mathbb R$.

Moreover, if there exists $f\in L^1(I)$ such that $\int_I f(x)dx=\alpha$ and $(\ref{2})$ holds for every $x\in I$, then $f$ is an $L^1(I)$-solution of equation $(\ref{e})$.
\end{theorem}

At the end of this paper it is worth noting that if an absolutely continuous  function $F\colon I\to [0,1]$ is nondecreasing, then it has a Radon-Nikodym derivative. However, not all absolutely continuous real functions have Radon-Nikodym derivatives (see \cite[Example 3.4]{KMa}).

\section*{Acknowledgement}
This research was supported by the University of Silesia Mathematics Department (Iterative Functional Equations and Real Analysis program).


\begin{thebibliography}{99}
\bibitem{A1995} L. Arnold,  Random Dynamical Systems, Lecture Notes in Math. 1609, Springer, Berlin, 1995.
\bibitem{B2009} K. Baron, On the convergence in law of iterates of random-valued functions, Aust. J. Math. Anal. Appl. 6 (2009), 1--9.
\bibitem{BKM} K. Baron, R. Kapica, J. Morawiec, On Lipschitzian solutions to an inhomogeneous linear iterative equation, Manuscript.
\bibitem{BK1977} K. Baron, M. Kuczma, Iteration of random-valued functions on the unit interval, Colloq. Math. 37 (1977), 263--269.
\bibitem{DH1998} T.B. Dinsenbacher, D.P. Hardin, Nonhomogeneous refinement equations, Wavelets, Multiwavelets, and their Applications, A. Aldroubi and E. Lin, eds., AMS, Providence, RI, 1998.
\bibitem{DH1999} T.B. Dinsenbacher, D.P. Hardin, Multivariate nonhomogeneous refinement equations, J. Fourier Anal. Appl. 5 (1999), 589--597.
\bibitem{GHM1994} J.S. Geronimo, D.P. Hardin, P.R. Massopust, Fractal functions and wavelet expansions based on several scaling functions, J. Approx. Theory 78 (1994), 373--401.
\bibitem{JJS2000} R.Q. Jia, Q.T. Jiang, Z.W. Shen, Distributional solutions of nonhomogeneous discrete and continuous refinement equations,  SIAM J. Math. Anal. 32  (2000),  420--434. 
\bibitem{JL2012} Q. Jiang, B. Li, Quad/triangle subdivision, nonhomogeneous refinement equation and polynomial reproduction, Math. Comput. Simulation  82  (2012), 2215--2237.
\bibitem{KM2008} R. Kapica, J. Morawiec, On a refinement type equation, J. Appl. Anal. 4 (2008), 251--257. 
\bibitem{KM2013}  R. Kapica, J. Morawiec, Refinement type equations: sources and results, Recent Developments in Functional Equations and Inequalities: Selected Topics, Banach Center Publ. 99 (2013), 87--110.
\bibitem{KM} R. Kapica, J. Morawiec, Inhomogeneous poly-scale refinement type equations and Markov operators with perturbations,  J. Fixed Point Theory Appl, DOI: 10.1007/s11784-015-0226-3.
\bibitem{KMa} R. Kapica, J. Morawiec, Inhomogeneous refinement equations with random affine maps, J. Difference Equ. Appl., DOI: 10.1007/s11784-015-0226-3.
\bibitem{KCG1990}  M. Kuczma, B. Choczewski, R. Ger, Iterative Functional Equations, Encyclopedia of Mathematics and its Applications, 32, Cambridge University Press, 1990.
\bibitem{MT1993} S.P. Meyn, R.L. Tweedie, Markov chains and stochastic stability, Communications and Control Engineering Series. Springer-Verlag London, 1993.
\bibitem{SN1996} G. Strang, T. Nguyen, Wavelets and Filter Banks, Wellesley, MA: Wellesley-Cambridge, 1996.
\bibitem{SZ1998} G. Strang, D.X. Zhou, Inhomogeneous refinement equations, J. Fourier Anal. Appl. 4 (1998), 733--747. 
\bibitem{S2001} Q. Sun, Compactly supported distributional solutions of nonstationary nonhomogeneous refinement equations, Acta Math. Sin. (Engl. Ser.) 17 (2001), 1--14.
\end{thebibliography}
\end{document}